\documentclass[10pt]{article}
\usepackage[T1]{fontenc}
\usepackage{amsmath,amsfonts,amsthm,mathrsfs,amssymb}
\usepackage{cite,bm,bbm}
\usepackage{multirow}
\usepackage{graphicx,float}
\usepackage{subfigure}
\usepackage{placeins}
\usepackage{color}
\usepackage{indentfirst}
\numberwithin{equation}{section}
\newcommand{\R}{{\mathbb R}}

\newcommand{\C}{{\mathbb C}}

\newcommand{\be}{\begin{eqnarray}}
\newcommand{\ben}{\begin{eqnarray*}}
\newcommand{\en}{\end{eqnarray}}
\newcommand{\enn}{\end{eqnarray*}}

\newcommand{\pa}{\partial}

\newcommand{\ov}{\overline}

\newtheorem{theorem}{Theorem}[section]
\newtheorem{lemma}[theorem]{Lemma}

\newtheorem{remark}[theorem]{Remark}

\definecolor{rot}{rgb}{1.000,0.000,0.000}
\definecolor{rot1}{rgb}{0.000,0.000,0.000}
\begin{document}
\renewcommand{\theequation}{\arabic{section}.\arabic{equation}}
\begin{titlepage}
\title{On a Robin-type non-singular coupling scheme for solving the wave scattering problems}

\author{
Xiaojuan Liu\thanks{School of Mathematical Sciences, University of Electronic Science and Technology of China, Chengdu, Sichuan 611731, China. Email:{\tt lxj$_-$math@163.com}},
Maojun Li\thanks{School of Mathematical Sciences, University of Electronic Science and Technology of China, Chengdu, Sichuan 611731, China. Email:{\tt limj@uestc.edu.cn}},
Tao Yin\thanks{LSEC, Institute of Computational Mathematics and Scientific/Engineering Computing, Academy of Mathematics and Systems Science, Chinese Academy of Sciences, Beijing 100190, China. Email:{\tt yintao@lsec.cc.ac.cn}}}
\date{}
\end{titlepage}
\maketitle

\begin{abstract}
This paper studies a non-singular coupling scheme for solving the acoustic and elastic wave scattering problems and its extension to the problems of Laplace and Lam\'e equations and the problem with a compactly supported inhomogeneity is also briefly discussed. Relying on the solution representation of the wave scattering problem, a Robin-type artificial boundary condition in terms of layer potentials whose kernels are non-singular, is introduced to obtain a reduced problem on a bounded domain. The wellposedness of the reduced problems and the a priori error estimates of the corresponding finite element discretization are proved. Numerical examples are presented to demonstrate the accuracy and efficiency of the proposed method.

\end{abstract}

\section{Introduction}
\label{sec:1}
\noindent Developing efficient numerical methods for solving the wave scattering problems, especially in an unbounded domain, is a fundamental task in many areas of scientific/engineering applications including computational optics, mechanics and electromagnetics, and geophysical exploration, to name a few. The method of using an artificial boundary condition (ABC) \cite{F84,HB99,HW13,ZMD21}, local or nonlocal \cite{BH00,KKS21,SPD06,VBA22}, to truncate the unbounded domain has been proved to be highly efficient and has been extensively studied for many years. In particular, the exact ABC can be derived by means of the so-called Dirichlet-to-Neumann (DtN) map which can be defined based on the Fourier series expansion of the solutions \cite{GYX17,G99,HNPX11,XY21} or the boundary integral operators \cite{DFFS22,ES08,HW08,HX11,YHX17} associated with the problems exterior to the artificial boundary. Use of the boundary integral operators is more flexible since it is not necessary to restrict the artificial boundary to be a circle or a sphere.

An alternative way to treat the wave scattering problem in an unbounded domain is to introduce the corresponding boundary integral equation (BIE) associated with the problems exterior to the artificial boundary and combine it with the boundary value problem inside the artificial boundary. Then relying on the FEM and boundary element method (BEM) for the discretization of the interior problem and the BIE, respectively, the well-known FEM-BEM coupling scheme results \cite{B07,F18,HS16,HM06,HSS17,MMP20,YRX18}. To overcome the calculation of the singular integrals involved in the FEM-BEM coupling scheme, a non-singular coupling idea is developed in \cite{XZH19} to propose an iterative algorithm for solving the exterior problem of the Laplace equation. This method utilizes two artificial boundaries and a special Green's function that satisfies a zero Dirichlet boundary condition on a circle/spherical boundary which locates inside another artificial boundary. However, this method can not be trivially extended to the wave scattering problems since unlike the problem of Laplace equation, it is difficult to compute such Green's functions for the problems of Helmholtz, Navier, or Maxwell equations.

This paper proposes a simpler non-singular coupling scheme for solving the acoustic and elastic wave scattering problems with Neumann boundary conditions, and it can be easily extended to the problems of Laplace and Lam\'e equations and the problem with a compactly supported inhomogeneity. The new idea is to use the exact solution representation of the original boundary value problem and then impose a Robin-type boundary condition on an artificial boundary. Due to the Neumann boundary condition, only Dirichlet data on the boundary of the obstacle is unknown in the ABC, and more importantly, the new ABC is non-singular and can be viewed as another type of the DtN map. Then the wellposedness of the reduced problem is proved based on the Fredholm alternative argument. Incorporating with the finite element method, the a priori error estimates in both $L^2$ and $H^1$ norms are derived. Numerical examples demonstrate the accuracy and efficiency of the proposed method. The application of the new method to the problems with Dirichlet boundary conditions or transmission boundary conditions is non-trivial since the Neumann data on the boundary of the obstacle is unknown implying that the ABC will be of Neumann-to-Neumann type. This will be left for future work.

This paper is organized as follows. The original wave scattering problems are described in Section~\ref{sec:2} and then Section~\ref{sec:3} discusses the proposed non-singular coupling method including the derivation of the non-singular artificial boundary condition (Section~\ref{sec:3.1}), the wellposedness of the reduced problem (Section~\ref{sec:3.2}) and the extension of this strategy to the problems of Laplace/Lam\'e equations (Section~\ref{sec:3.3}) and the problem with a compactly supported inhomogenity (Section~\ref{sec:3.4}). The finite element approximation of the reduced boundary value problem is studied in Section~\ref{sec:4} and a priori error estimates are proved. Finally, the numerical experiments are shown in Section \ref{sec:5} to demonstrate the accuracy and efficiency of the presented approach.

\section{Wave scattering problems}
\label{sec:2}
\noindent As shown in Fig.~\ref{model}(a), let $\Omega\subset\R^2$ be a bounded and simply connected domain with, for simplicity, smooth boundary $\Gamma:=\pa\Omega$ and denoted by $\Omega^{c}:=\mathbb{R}^{2}\backslash\ov{\Omega}$ the unbounded exterior domain. In this work, the following model exterior acoustic and elastic wave scattering problems for $u:\Omega^c\rightarrow \C^\sigma$ with Neumann-type boundary conditions will be considered:
\begin{align}
    \mathcal{L}u&=0 \quad\mbox{in}\quad\Omega^{c}, \label{Lequ-1}\\
    Tu&=g \quad\mbox{on}\quad\Gamma, \label{Lequ-2}
  \end{align}
where $\mathcal{L}$ represents one of the following linear differential operators:
\begin{equation}\label{Lequ-3}
  \mathcal{L}u=
  \begin{cases}
  \Delta u+k^{2}u &(\mathrm{Acoustic}, \sigma=1),\cr
  \Delta^{*}u+ \rho\omega^{2}u &(\mathrm{Elastic}, \sigma=2).
  \end{cases}
\end{equation}
Here, $\omega>0$ denotes the angular frequency, $k=\omega/c$ denotes the acoustic wave number with $c>0$ being the acoustic wave speed, and $\rho>0$ is a density constant. The Lam\'e operator $\Delta^{*}$ is defined by
\begin{equation*}
\Delta^{*}=\mu \Delta + (\lambda + \mu)\nabla\nabla\cdot,
\end{equation*}
where $\lambda, \mu$ denotes the Lam\'e parameters satisfying $\mu>0, \lambda+\mu>0$. For the acoustic case, the Neumann boundary operator $T$ is given by $Tu=\frac{\partial u}{\partial \bm{n}}$ where ${\bm n}$ denotes the unit outward normal to the boundary $\Gamma$. Alternatively, for the elastic case, $T$ is defined as
\ben
Tu= 2\mu\frac{\partial u}{\partial\bm{n}} + \lambda\bm{n}~\nabla\cdot u + \mu\bm{n}^{\bot}(\partial_{2}u_{1}-\partial_{1}u_{2}),
\enn
where $\bm{n}^{\bot}:=(-n_{2},n_{1})^{\top}$ with $\bm{n}=(n_1,n_2)^\top$. Given an incident field $u^i$ satisfying $\mathcal{L}u^i=0$ in $\R^2$, for example, the plane wave, the solution $u$ represents the scattered field and then the boundary data $g$ on $\Gamma$ will be given by $g=-Tu^i$.

Moreover, the ensurance of the wellposedness of the exterior boundary value problems requires the following suitable boundary conditions at infinity:
\be
\label{RC-acoustic}
\lim_{r\rightarrow\infty} r^{1/2}\left(\frac{\partial u}{\partial r}-iku\right)=0,\quad r=|x|,
\en
for the acoustic case and
\be
\label{RC-elastic}
\lim_{r\rightarrow\infty} r^{1/2}\left(\frac{\partial u_t}{\partial r}-ik_tu_t\right)=0,\quad r=|x|,\;\;t=p,s,
\en
for the elastic case where $u_{p}$ and $u_{s}$ represent the compressional and the shear waves with wave numbers $k_{p}$ and $k_{s}$, respectively, which are given as
\begin{equation*}
k_{p}^{2} = \frac{\rho\omega^{2}}{\lambda + 2\mu},\qquad k_{s}^{2} =  \frac{\rho\omega^{2}}{\mu}.
\end{equation*}

\begin{theorem}\cite{CK19} \label{Theorem1}
The exterior boundary value problem (\ref{Lequ-1})-(\ref{Lequ-2}) is uniquely solvable.
\end{theorem}

\begin{remark}
The discussed non-singular coupling scheme in the following sections can be easily extended to the problems with the Robin boundary condition as well as the three-dimensional problems and the electromagnetic scattering problems, and the corresponding investigation follows analogously.
\end{remark}

\begin{figure}[htbp]
 \centering
 \begin{tabular}{cc}
 \includegraphics[scale=0.25]{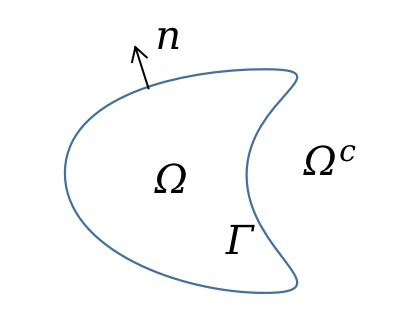} &
 \includegraphics[scale=0.25]{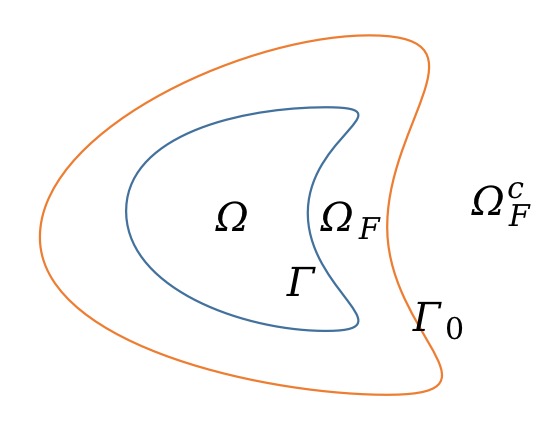} \\
 (a) Problem \eqref{Lequ-1}-\eqref{Lequ-2} & (b) Problem \eqref{Tru-equ1}-\eqref{Tru-equ3}
 \end{tabular}
\caption{Configuration of the computational domain}
 \label{model}
 \end{figure}
\section{Non-singular coupling method}
\label{sec:3}

\noindent In this section, we propose a general non-singular coupling method for solving the acoustic and elastic scattering problems (\ref{Lequ-1})-(\ref{Lequ-2}) together with the corresponding wellposedness analysis. The application of a simplified version to the problems of Laplace and Lam\'e equations (i.e. $\omega=0$) as well as an extension to the problems with a compactly supported inhomogeneity will also be discussed.

\subsection{Truncation with non-singular integrals}
\label{sec:3.1}

\noindent Let $\Gamma_{0}$ be a smooth closed curve which is large enough to enclose the entire region $\Omega$, see Fig.~\ref{model}(b). Then the artificial boundary $\Gamma_{0}$ decomposes the exterior domain $\Omega^{c}$ into two subdomains $\Omega_{F}$ and $\Omega_{F}^{c}$ respectively, where $\Omega_{F}$ is the region between $\Gamma$ and $\Gamma_{0}$, and $\Omega_{F}^{c}=\mathbb{R}^{2}\backslash \overline{\Omega\cup\Omega_{F}}$ is the unbounded exterior region. The two boundaries $\Gamma$ and $\Gamma_0$ are well separated such that $\min_{x\in\Gamma_0, y\in\Gamma}|x-y|\geq c_0 $ with $c_0$ being a positive constant.

Let $u^*(x,y)$ be the fundamental solution of the Helmholtz or Navier equations in $\R^2$ which are given by
\begin{equation*}
 u^{\ast}(x,y)= \frac{i}{4}H_{0}^{(1)}(k|x-y|),
\end{equation*}
for the acoustic case and
\begin{equation*}
 u^{\ast}(x,y)= \frac{i}{4\mu}H_{0}^{(1)}(k_{s}|x-y|)\mathbf{I}+\frac{i}{4\rho\omega^2} \nabla_x\nabla_x \left[ H_{0}^{(1)}(k_{s}|x-y|)- H_{0}^{(1)}(k_{p}|x-y|)\right],
\end{equation*}
for the elastic case where $H_{0}^{(1)}$ denotes the Hankel function of the first kind with order zero and $\mathbf{I}$ is the  $2\times 2$ identity matrix. It is known that the solutions $u$ in $\Omega^c$ admit the representation
\be
\label{solrep}
u(x)=\mathcal{B}(\Gamma,u)(x):= D_{\Gamma}(u)(x)-S_{\Gamma}(Tu)(x),\quad x\in\Omega^c,
\en
where the single-layer operator $S_{\Gamma}$ and double-layer operator $D_{\Gamma}$ are defined as
\begin{equation*}
  \begin{aligned}
 &S_{\Gamma}(Tu)(x)=\int_{\Gamma} u^{\ast}(x,y)T_{y}u(y)ds_{y},\quad x\in\Omega^c,\\
 &D_{\Gamma}(u)(x)=\int_{\Gamma} (T_{y}u^{\ast}(x,y))^{\top}u(y)ds_{y},\quad x\in\Omega^c,
  \end{aligned}
 \end{equation*}
respectively. In particular, we get on $\Gamma_0$ that
\ben
u=\mathcal{B}(\Gamma,u)\quad\mbox{on}\quad\Gamma_0,
\enn
\ben
Tu=T\mathcal{B}(\Gamma,u)\quad\mbox{on}\quad\Gamma_0,
\enn
and therefore,
\be
\label{ABC}
\Lambda_{\alpha}(u-D_{\Gamma}(u))&= -\Lambda_{\alpha}(S_{\Gamma}(g))\quad\mbox{on}\quad\Gamma_0,
\en
where $\Lambda_{\alpha}v=Tv-i\alpha v$ with $\alpha\ne 0$ being an impedance coefficient. Noting that the integrals in (\ref{ABC}) are only for $x\in\Gamma_0$ and $y\in\Gamma$, and thus the boundary condition (\ref{ABC}) is non-singular.

Utilizing the non-singular boundary condition (\ref{ABC}) results into the following truncated boundary value problem: Given $g\in H^{-1/2}(\Gamma)^\sigma$, find $u\in H^1(\Omega_{F})^\sigma$ such that
\begin{align}
    \mathcal{L}u&=0\quad\mbox{in}\quad\Omega_{F}, \label{Tru-equ1}\\
    Tu(x)&=g\quad\mbox{on}\quad\Gamma, \label{Tru-equ2}\\
    \Lambda_{\alpha}(u-D_{\Gamma}(u))&= -\Lambda_{\alpha}(S_{\Gamma}(g))\quad\mbox{on}\quad\Gamma_0.\label{Tru-equ3}
  \end{align}
The uniqueness of the solution to the problem \eqref{Tru-equ1}-\eqref{Tru-equ3} is given in the following theorem.

\begin{theorem}\label{Thm-uniqueness}
The truncated problem \eqref{Tru-equ1}-\eqref{Tru-equ3} has at most one solution.
\end{theorem}

\begin{proof}
Thanks to the uniqueness of the original boundary value problem \eqref{Lequ-1}-\eqref{Lequ-2}, it is only necessary to prove the equivalence of the solutions to the problems \eqref{Lequ-1}-\eqref{Lequ-2} and \eqref{Tru-equ1}-\eqref{Tru-equ3} in $\Omega_F$ with $g=0$.\\

\begin{itemize}
\item [$\bullet$]{\bf From \eqref{Lequ-1}-\eqref{Lequ-2} to \eqref{Tru-equ1}-\eqref{Tru-equ3}.} This direction is obvious.\\

\item [$\bullet$]{\bf From \eqref{Tru-equ1}-\eqref{Tru-equ3} to \eqref{Lequ-1}-\eqref{Lequ-2}.} Let $u$ be the solutions to the problem \eqref{Tru-equ1}-\eqref{Tru-equ3} with $g=0$. It is sufficient to prove that $u=\mathcal{B}(\Gamma,u)$ in $\Omega_F$ which can be continuously extended to the whole $\Omega^c$ as $\widetilde{u}=\mathcal{B}(\Gamma,u)$ in $\Omega^c$ which satisfies \eqref{Lequ-1}-\eqref{Lequ-2} with $g=0$. In fact, we know from the Green's presentation that
\ben\label{uni-equ2}
  u=-\mathcal{B}(\Gamma_{0}, u)+\mathcal{B}(\Gamma, u) \quad\mbox{in}\quad\Omega_F.
\enn
Let $\hat{u}=-\mathcal{B}(\Gamma_{0}, u)$ in $\Omega_F\cup\ov{\Omega}$. Obviously, $\mathcal{L}\hat u=0$ in $\Omega_F\cup\ov{\Omega}$. Moreover, the boundary condition \eqref{Tru-equ3} yields that
\ben
\Lambda_{\alpha}\hat{u}=0 \quad\mbox{on}\quad\Gamma_0.
\enn
Then it follows from the variational approach and Holmgren's  theorem that $\hat{u}=0$ in $\Omega_F\cup\ov{\Omega}$. Therefore, $u=\mathcal{B}(\Gamma,u)$ in $\Omega_F$.
\end{itemize}
This completes the proof.
\end{proof}

\subsection{Variational formulation and wellposedness}
\label{sec:3.2}

\noindent Now, we study the variational problem of \eqref{Tru-equ1}-\eqref{Tru-equ3} as follows: Find $u\in H^{1}(\Omega_{F})^\sigma$ such that
\begin{equation}\label{Var-equ1}
       \begin{array}{l l}
       a(u,v)+b(u,v)=\ell(v),~~~~~\forall v\in H^{1}(\Omega_{F})^\sigma,\\
  \end{array}
\end{equation}
where
\begin{equation}\label{Var-equ2}
  a(u,v):=\left\{
    \begin{aligned}
  &\int_{\Omega_{F}}\nabla u \cdot \nabla \bar{v}\mathrm{d}{x}-k^{2}\int_{\Omega_{F}}u\bar{v}d x~~~(\mathrm{Acoustic}),\\
   &\int_{\Omega_F} \left(\mu\nabla u:\nabla \bar{v} + (\lambda + \mu)~(\nabla\cdot u)~(\nabla\cdot \bar{v})-\rho\omega^{2} u\cdot \bar{v}\right)d x~(\mathrm{Elastic}),\\
  \end{aligned}
  \right.
\end{equation}

\begin{equation}\label{Var-equ3}
   b(u,v):=-i\alpha\int_{\Gamma_0}u\cdot \ov{v}ds- \int_{\Gamma_0} \Lambda_{\alpha} D_{\Gamma}(u)\cdot \ov{v}ds,
\end{equation}
and
\begin{equation}\label{Var-equ4}
  \ell(v):= -\int_{\Gamma}g\cdot \ov{v}ds- \int_{\Gamma_0} \Lambda_\alpha S_{\Gamma}(g)\cdot \ov{v}ds.
\end{equation}

In the following, we always write $A\lesssim B$ (resp. $A\gtrsim B$) for the inequality $A\le c B$ (resp. $A\ge cB$) with $c>0$ being a constant.

\begin{lemma}\label{Thm-garding}
The sesquilinear form $a(\cdot,\cdot)+b(\cdot, \cdot)$ in \eqref{Var-equ1} is bounded as
\begin{equation}\label{Var-bound}
\left|a(u,v)+b(u, v)\right|\lesssim \Vert u \Vert_{H^{1}(\Omega_{F})^\sigma}\Vert v \Vert_{H^{1}(\Omega_{F})^\sigma},\quad\forall u,v\in H^{1}(\Omega_{F})^\sigma,
\end{equation}
and satisfies a G{\aa}rding's inequality in the form
\begin{equation}\label{Var-garding}
      \begin{array}{l l}
\mathrm{Re}\{a(v,v)+b(v, v)\}+ c\Vert v \Vert_{H^{1-\epsilon}(\Omega_{F})^\sigma}^2\gtrsim \Vert v \Vert^{2}_{H^{1}(\Omega_{F})^\sigma},
 \end{array}
\end{equation}
where $c>0$ and $\epsilon \in (0,1/2)$ are constants.
\end{lemma}
\begin{proof}
It easily follows from the Cauchy-Schwarz inequality that
\be
\label{garding1}
  |a(u,v)|\lesssim \Vert u \Vert_{H^{1}(\Omega_{F})^\sigma}\Vert v \Vert_{H^{1}(\Omega_{F})^\sigma}.
\en
Moreover, for the acoustic case,
\be
\label{garding2}
a(v,v)=\Vert v \Vert_{H^{1}(\Omega_{F})^\sigma}^2- (1+k^2)\Vert v \Vert_{L^2(\Omega_{F})^\sigma}^2.
\en
For the elastic case, it follows from the Korn's inequality that
there exists a constant $c_1>0$ such that
\be
\label{garding3}
a(v,v)+c_1\Vert v \Vert_{L^2(\Omega_{F})^\sigma}^2\gtrsim \Vert v \Vert_{H^{1}(\Omega_{F})^\sigma}^2.
\en
Noting that $\min_{y\in\Gamma, x\in\Gamma_0}|x-y|\geq c_0>0$, there exists a constant $c_2>0$ such that
\ben
\left|T_{y} u^{\ast}(x,y) \right| \le c_2,\quad x\in\Gamma_0, y\in\Gamma,
\enn
and
\ben
\left|T_{x}(T_{y}u^{\ast}(x,y))^{\top}\right| \le c_2,\quad x\in\Gamma_0, y\in\Gamma.
\enn
Then we get
\ben
|b(u,v)|&&\le |\alpha|\Vert u\Vert_{L^2(\Gamma_0)^\sigma}\Vert v\Vert_{L^2(\Gamma_0)^\sigma}  + c_2(1+|\alpha|) \Vert u\Vert_{L^2(\Gamma)^\sigma}\Vert v\Vert_{L^2(\Gamma_0)^\sigma}\nonumber\\
&& \lesssim (\Vert u\Vert_{L^2(\Gamma_0)^\sigma}+\Vert u\Vert_{L^2(\Gamma)^\sigma})\Vert v \Vert_{L^2(\Gamma_0)^\sigma}.
\enn
Therefore,
\be
\label{garding4}
|b(u,v)|\lesssim \Vert u\Vert_{H^{1}(\Omega_{F})^\sigma}\Vert v\Vert_{H^{1}(\Omega_{F})^\sigma},
\en
and
\be
\label{garding5}
\mathrm{Re}\{b(v,v)\} +c_3 \Vert v\Vert_{H^{1-\epsilon}(\Omega_{F})^\sigma}^2 \ge 0,
\en
with $c_3>0$ and $\epsilon \in (0,1/2)$ being constants. Then \eqref{Var-bound} follows from (\ref{garding1}) and (\ref{garding4}), and \eqref{Var-garding} follows from (\ref{garding2})-(\ref{garding3}) and (\ref{garding5}), which completes the proof.
\end{proof}

Then Theorem~\ref{Thm-uniqueness} and Lemma~\ref{Thm-garding} together with the Fredholm alternative theorem yield the following wellposedness result.

\begin{theorem}\label{Thm-wellposedness}
The variational problem (\ref{Var-equ1}) admits a unique solution $u\in H^{1}(\Omega_{F})^\sigma$. Moreover, the inf-sup condition
\ben
\sup_{0\ne v\in H^{1}(\Omega_{F})^\sigma} \frac{|a(u,v)+b(u,v)|}{\Vert v\Vert_{H^{1}(\Omega_{F})^\sigma}} \gtrsim \Vert u\Vert_{H^{1}(\Omega_{F})^\sigma},\quad\forall u\in H^{1}(\Omega_{F})^\sigma,
\enn
and the priori bound
\ben
\Vert u\Vert_{H^{1}(\Omega_{F})^\sigma}\lesssim \Vert g\Vert_{H^{-1/2}(\Gamma)^\sigma},
\enn
hold.
\end{theorem}

\subsection{Extension I. problems of Laplace and Lam\'e equations}
\label{sec:3.3}
\noindent
For the problems of Laplace and Lam\'e equations, i.e., the linear differential operator $\mathcal{L}$ is given by
\ben
\mathcal{L}u=\begin{cases}
\Delta u & (\mbox{Laplace}, \sigma=1), \cr
\Delta^*u & (\mbox{Lam\'e}, \sigma=2).
\end{cases}
\enn
Let $u^*(x,y)$ be the fundamental solution of the Laplace or Lam\'e equations in $\R^2$ which are given by
\begin{equation*}
 u^{\ast}(x,y)=\frac{1}{2\pi}\log|x-y|,
\end{equation*}
for the Laplace case and
\begin{equation*}
 u^{\ast}(x,y)=\frac{\lambda+3\mu}{4\pi\mu(\lambda+2\mu)}\left[-\log|x-y|\mathbf{I}+\frac{\lambda+\mu}{(\lambda+3\mu)}\frac{1}{|x-y|^{2}}(x-y)(x-y)^{T}\right],\\
\end{equation*}
for the Lam\'e case. The Neumann boundary operator $T$ for the Laplace (resp. Lam\'e) equation is the same as that for the acoustic (resp. elastic) case and the following boundary conditions at infinity are imposed:
\be
\label{RC-Laplace}
 u=O(1)\quad\mathrm{as}\quad r=|x|\rightarrow\infty,
\en
for Laplace case and
\be
\label{RC-Lame}
u(x)=-u^{\ast}(x,0)\bm{\Sigma}+\bm{w}(x)+O(|x|^{-1})\quad\mathrm{as}\quad |x|\rightarrow\infty,
\en
for Lam\'e case, where $\bm{w}(x)$ is a rigid motion defined by
\be
\bm{w}(x)=\bm{a}+b(-x_{2}, x_{1})^{\top},
\en
where $\bm{a}$ and $\bm{\Sigma}$ are constant vectors and $b$ is a constant. For the Laplace case, we additionally assume that $\int_{\Gamma} gds=0$. Let $\mathbb{G}$ be the space of constants for the Laplace case and be the space of rigid displacement defined as
\ben
\mathbb{G}:=\{\bm{a}+b(-x_{2}, x_{1})^{\top},\quad \bm{a}\in \R^2, b\in\R\}.
\enn
In the following, we always consider the solutions in $\mathbb{H}:=H^1(\Omega_F)^\sigma\backslash\mathbb{G}$.

The solution representation (\ref{solrep}) gives
\be
\label{ABC-2}
Tu=T\mathcal{B}(\Gamma,u)\quad\mbox{on}\quad\Gamma_0.
\en
Then we can get the following truncated boundary value problem: Given $g\in H^{-1/2}(\Gamma)^\sigma$, find $u\in \mathbb{H}$ such that
\begin{align}
    \mathcal{L}u&=0\quad\mbox{in}\quad\Omega_{F}, \label{Tru-equ1-2}\\
    Tu(x)&=g\quad\mbox{on}\quad\Gamma, \label{Tru-equ2-2}\\
    Tu-TD_{\Gamma}(u)&= -TS_{\Gamma}(g)\quad\mbox{on}\quad\Gamma_0,\label{Tru-equ3-2}
  \end{align}
in which we just use a Neumann-type boundary condition on $\Gamma_0$, i,e. $\alpha=0$. The uniqueness of the solution to the problem \eqref{Tru-equ1-2}-\eqref{Tru-equ3-2} is given in the following theorem.

\begin{theorem}\label{Thm-uniqueness-2}
The truncated problem \eqref{Tru-equ1-2}-\eqref{Tru-equ3-2} has at most one solution.
\end{theorem}
\begin{proof}
The proof is analogous to that for Theorem~\ref{Thm-uniqueness}. The only difference is that of $\hat u=-\mathcal{B}(\Gamma_0,u)$, it holds that $\mathcal{L}\hat u=0$ in $\Omega_F\cup\ov{\Omega}$ and $T\hat u=0$ on $\Gamma_0$ which implies that $\hat u=0$ in $\Omega_F\cup\ov{\Omega}$. This completes the proof.
\end{proof}

The variational problem of \eqref{Tru-equ1-2}-\eqref{Tru-equ3-2} reads: Find $u\in \mathbb{H}$ such that
\begin{equation}\label{Var-equ1-2}
       \begin{array}{l l}
       a(u,v)+b(u,v)=\ell(v),~~~~~\forall v\in H^{1}(\Omega_{F})^\sigma,\\
  \end{array}
\end{equation}
where
\begin{equation}\label{Var-equ2-2}
  a(u,v):=\left\{
    \begin{aligned}
  &\int_{\Omega_{F}}\nabla u \cdot \nabla \bar{v}\mathrm{d}{x}~~~(\mbox{Laplace}),\\
   &\int_{\Omega_F} \left(\mu\nabla u:\nabla \bar{v} + (\lambda + \mu)~(\nabla\cdot u)~(\nabla\cdot \bar{v})\right)d x~(\mbox{Lam\'e}),\\
  \end{aligned}
  \right.
\end{equation}
\begin{equation}\label{Var-equ3-2}
   b(u,v):= -\int_{\Gamma_{0}}TD_{\Gamma}(u)\cdot \ov{v}ds,
\end{equation}
and
\begin{equation}\label{Var-equ4-2}
  \ell(v):= -\int_{\Gamma_{0}}TS_{\Gamma}(g)\cdot \ov{v}ds. \\
\end{equation}
Then it follows from Lemma~\ref{Thm-garding} that the sesquilinear form $a(\cdot,\cdot)+b(\cdot, \cdot)$ in \eqref{Var-equ1-2} is bounded as
\begin{equation}\label{Var-bound-2}
\left|a(u,v)+b(u, v)\right|\lesssim \Vert u\Vert_{H^{1}(\Omega_{F})^\sigma}\Vert v\Vert_{H^{1}(\Omega_{F})^\sigma},\quad\forall u,v\in H^{1}(\Omega_{F})^\sigma,
\end{equation}
and satisfies a G{\aa}rding's inequality in the form
\begin{equation}\label{Var-garding-2}
      \begin{array}{l l}
\mathrm{Re}\{a(v,v)+b(v, v)\}+ c\Vert v\Vert_{H^{1-\epsilon}(\Omega_{F})^\sigma}^2\gtrsim \Vert v\Vert^{2}_{H^{1}(\Omega_{F})^\sigma},
 \end{array}
\end{equation}
where $c>0$ and $\epsilon \in (0,1/2)$ are constants. As a consequence, analogous wellposedness result as Theorem~\ref{Thm-wellposedness} can be obtained.

\subsection{Extension II. problems with a compactly supported inhomogeneity}
\label{sec:3.4}
\noindent
Now we consider the acoustic scattering problem in an inhomogeneous medium which is modeled by
\begin{align}
\label{ProblemInhomogeneous}
\Delta u+k^2nu&=0\quad\mbox{in}\quad\R^2,
\end{align}
where $n\not\equiv 1$ is the refractive index and $n-1$ is compactly supported such that $n(x)=1, |x|>R$ for some $R>0$, and $u=u^i+u^s$ with $u^i$ being the incident field satisfying $\Delta u^i+k^2u^i=0$ in $\R^2$ and $u^s$ being the scattered field which satisfied the radiation condition (\ref{RC-acoustic}). In particular, we assume that $\mathrm{Re}\{n\}>0$ and $\mathrm{Im}\{n\}\ge0$. The uniqueness of this boundary value problem is given in \cite{CK19}.

Let $\Gamma_0$ be a closed artificial boundary and denote by $\Omega_F$ and $\Omega_F^c$ the interior and exterior domain, respectively. The selection of $\Gamma_0$ is such that $n-1$ is compactly supported in a subset of $\Omega_F$. It follows from the Lippmann-Schwinger equation that the solution $u$ in $\Omega_F^c$ admits the representation
\ben
u=u^i-k^2V_{\Omega_F}((1-n)u) \quad\mbox{in}\quad\R^2,
\enn
where the volume potential $V$ is defined as
\ben
V_{\Omega}(\varphi)(x)=\int_{\Omega} u^*(x,y)\varphi(y)dy.
\enn
Now we can introduce a non-singular boundary condition on $\Gamma_0$ as
\ben
\Lambda_\alpha(u)+k^2\Lambda_\alpha V_{\Omega_F}((1-n)u)=\Lambda_\alpha(u^i) \quad\mbox{on}\quad\Gamma_0,
\enn
and then arrive at the following truncated boundary value problem: Given $u^i\in H_{loc}^1(\R^2)$ (thus, $\Lambda_\alpha(u^i)\in H^{-1/2}(\Gamma_0)$), find $u\in H^1(\Omega_F)$ such that
\begin{align}
  \label{InTru-equ1}
  \Delta u+k^2nu&=0\quad\mbox{in}\quad\Omega_F,\\
  \label{InTru-equ2}
  \Lambda_\alpha(u)+k^2\Lambda_\alpha V_{\Omega_F}((1-n)u)&=\Lambda_\alpha(u^i) \quad\mbox{on}\quad\Gamma_0.
  \end{align}

\begin{theorem}\label{Thm-uniqueness-3}
The truncated problem (\ref{InTru-equ1})-(\ref{InTru-equ2}) has at most one solution.
\end{theorem}
\begin{proof}
The proof follows analogous to Theorem~\ref{Thm-uniqueness} and thus is omitted here.
\end{proof}

The variational problem of \eqref{InTru-equ1}-\eqref{InTru-equ2} reads: Find $u\in H^{1}(\Omega_{F})^\sigma$ such that
\begin{equation}\label{Var-equ1-3}
       \begin{array}{l l}
       a(u,v)+b(u,v)=\ell(v),~~~~~\forall v\in H^{1}(\Omega_{F})^\sigma,\\
  \end{array}
\end{equation}
where
\begin{equation}\label{Var-equ2-3}
  a(u,v):=
    \begin{aligned}
  \int_{\Omega_{F}}\nabla u \cdot \nabla \bar{v}\mathrm{d}{x}-k^{2}\int_{\Omega_{F}}nu\bar{v}d x,
  \end{aligned}
\end{equation}
\begin{equation}\label{Var-equ3-3}
   b(u,v):=-i\alpha\int_{\Gamma_0}u\cdot \ov{v}ds+\int_{\Gamma_0} k^{2}\Lambda_\alpha V_{\Omega_F}((1-n)u)\cdot \ov{v}ds,
\end{equation}
and
\begin{equation}\label{Var-equ4-3}
  \ell(v):= \int_{\Gamma_0} \Lambda_\alpha(u^i)\cdot \ov{v}ds.
\end{equation}
It can proved that the variational problem (\ref{Var-equ1-3}) admits a unique solution $u\in H^1(\Omega_F)^\sigma$.
\section{Finite element approximation}
\label{sec:4}
\noindent
In this section, we consider the finite element discretization of the variational problem (\ref{Var-equ1}) and derive a priori error estimates. The corresponding analysis for the variational problems (\ref{Var-equ1-2}) and (\ref{Var-equ1-3}) follows analogously and thus is omitted. For simplicity, here we employ the standard finite element method based on continuous piecewise linear basis on a triangulation $\mathcal{T}_h$ of $\Omega_F$:
\begin{equation*}
  \mathcal{T}_{h}=\{K|\cup \overline{K}=\overline{\Omega_{F}}, \mathrm{diameter}(K)\leq h\},
  \end{equation*}
with $h>0$ being the maximum diameter. To match the smooth boundaries $\Gamma_0$ and $\Gamma$, there must be some 'pi-shaped' elements whose one edge is located on $\Gamma_0$ or $\Gamma$. But here, we assume that $\Gamma$ and $\Gamma_0$ are approximated well by polygons and omit the discussion of this approximation error. Let $S_h\subset H^1(\Omega_F)^\sigma$ be the finite element space and the approximation property is satisfied:
\ben
\inf_{\xi\in S_h}\left\{ \Vert v-\xi\Vert_{L^2(\Omega_{F})^\sigma}+ h|v-\xi|_{H^1(\Omega_{F})^\sigma}\right\}\le ch^{1+\alpha}\Vert v\Vert_{H^{1+\alpha}(\Omega_{F})^\sigma},
\enn
for all $v\in H^{1+\alpha}(\Omega_{F})^\sigma, 0\le \alpha\le 1$ is satisfied. The Galerkin formulation of the problem \eqref{Var-equ1} reads: find $u_{h}\in S_{h}$ such that
\begin{equation}\label{EFin-equ1}
a(u_{h},v_{h})+b(u_{h},v_{h})=\ell(v_{h}),~\forall v_{h}\in  S_{h}.
\end{equation}
It can be derived from \cite{HNPX11, HSS17} that the sesquilinear form $a(\cdot,\cdot)+b(\cdot,\cdot)$ satisfies a discrete inf-sup condition
\begin{equation}\label{Fin-equ2}
\sup_{0\neq w_{h}\in  S_{h}}\frac{|a(v_{h},w_{h})+b(v_{h},w_{h})|}{\Vert w_{h}\Vert_{H^1(\Omega_F)^\sigma}}\gtrsim \Vert v_{h}\Vert _{H^1(\Omega_F)^\sigma},~~~\forall v_{h}\in  S_{h}.
\end{equation}
\begin{theorem}
\label{theorem-2}
The discrete variational problem \eqref{EFin-equ1} admits a unique solution $u_{h}\in S_h$. Moreover, if $u\in  H^{1+\alpha_{1}}(\Omega_{F})^\sigma$ for some $0\leq \alpha_{1}\leq1$, there exists a constant $h_{0}>0$ such that for any $0 < h \leq h_{0}$
\begin{equation}\label{Fin-equ3}
\Vert u-u_{h}\Vert_{H^{1}(\Omega_{F})^\sigma}\lesssim  h^{\alpha_{1}}\Vert u\Vert_{H^{1+\alpha_{1}}(\Omega_{F})^\sigma}.
\end{equation}
\end{theorem}
\begin{proof}
According to the discrete inf-sup condition \eqref{Fin-equ2} and the Galerkin orthogonality from \eqref{Var-equ1} and \eqref{EFin-equ1}, for $\forall v_{h}\in S_{h}$, we get
 \begin{equation}\label{Fin-equ4}
\begin{aligned}
\Vert u_{h}-v_{h}\Vert_{ H^{1}(\Omega_{F})^\sigma}&\lesssim \sup_{0\neq w_{h}\in S_{h}}\frac{|a(u_{h}-v_{h},w_{h})+b(u_{h}-v_{h},w_{h})|}{\Vert w_{h}\Vert_{ H^{1}(\Omega_{F})^\sigma}}\\
&=\sup_{0\neq w_{h}\in S_{h}}\frac{|a(u-v_{h},w_{h})+b(u-v_{h},w_{h})|}{\Vert w_{h}\Vert_{ H^{1}(\Omega_{F})^\sigma}} \nonumber \\
& \lesssim \Vert u-v_{h}\Vert_{H^{1}(\Omega_{F})^\sigma}.
\end{aligned}
\end{equation}
Then the triangular inequality means
\begin{equation*}\label{Fin-equ5}
\begin{aligned}
\Vert u-u_{h}\Vert_{H^{1}(\Omega_{F})^\sigma}
&\leq\Vert u-v_{h}\Vert_{ H^{1}(\Omega_{F})^\sigma}+\Vert v_{h}-u_{h}\Vert_{H^{1}(\Omega_{F})^\sigma} \nonumber \\
&\lesssim \Vert u-v_{h}\Vert_{H^{1}(\Omega_{F})^\sigma}.
\end{aligned}
\end{equation*}
Therefore,
\begin{equation}\label{Fin-equ6}
 \Vert u-u_{h}\Vert_{H^{1}(\Omega_{F})^\sigma}\lesssim \inf_{\xi\in S_h} \Vert u-\xi\Vert_{H^1(\Omega_{F})^\sigma}\lesssim h^{\alpha_{1}}\Vert u\Vert_{H^{1+\alpha_{1}}(\Omega_{F})^\sigma},
  \end{equation}
which completes the proof.
\end{proof}

To derive the $L^{2}$-error estimate, we define the following dual problem:
\begin{align}
\mathcal{L}w &=  e \quad\mathrm{in}\quad\Omega_{F}, \label{Fin-equ7-1}\\
\Lambda_\alpha^*w &= 0  \quad\mathrm{on}\quad\Gamma_0,\label{Fin-equ7-2}\\
Tw+i\alpha TS_{\Gamma_0}^*w+ N_{\Gamma_0}^* w &= 0 \quad\mathrm{on}\quad\Gamma,\label{Fin-equ7-3}
\end{align}
and assume that the following regularity holds
\begin{equation}\label{Fin-equ8}
\Vert w\Vert_{H^{1+\alpha_{2}}(\Omega_{F})^\sigma} \lesssim\Vert e\Vert_{L^{2}(\Omega_{F})^\sigma}.
\end{equation}
Here, $\Lambda_\alpha^*v=Tv+i\alpha v$ and the integral operators $S_{\Gamma_0}^*$ and $N_{\Gamma_0}^*$ are defined as
\begin{equation*}
  \begin{aligned}
 &S_{\Gamma_0}^*(w)(x)=\int_{\Gamma} \ov{u^{\ast}(x,y)}w(y)ds_{y},\quad x\in\Gamma_{0},\\
 &N_{\Gamma_0}^*(w)(x)=\int_{\Gamma} T_x(T_{y}\ov{u^{\ast}(x,y)})^{\top}w(y)ds_{y},\quad x\in\Gamma_{0}.
  \end{aligned}
 \end{equation*}

\begin{theorem}
\label{theorem-3}
If $u\in  H^{1+\alpha_{1}}(\Omega_{F})^\sigma$ for some $\frac{1}{2}< \alpha_{1}\leq1$, there exists a constant $h_{0}>0$ such that for any $0 < h \leq h_{0}$ and $\frac{1}{2}< \alpha_{2}\leq1$
\begin{equation}\label{Fin-equ7}
\Vert u-u_{h}\Vert_{L^{2}(\Omega_{F})^\sigma}\lesssim h^{\alpha_{1}+\alpha_2}\Vert u\Vert_{H^{\alpha_{1}}(\Omega_{F})^\sigma}.
\end{equation}
\end{theorem}
\begin{proof}
Multiplying the conjugates of \eqref{Fin-equ7-1} with $u-u_{h}$ and integrating by parts yield
\begin{equation*}
  \begin{aligned}
&\quad (u-u_{h},e)_{L^{2}(\Omega_{F})^\sigma} \\
&=a(u-u_h,w) +\int_{\Gamma} (u-u_h)\cdot T\ov{w}ds- \int_{\Gamma_0} (u-u_h)\cdot T\ov{w}ds\\
&= a(u-u_h,w) +\int_{\Gamma} (u-u_h)\cdot (\ov{-i\alpha TS_{\Gamma_0}^*w- N_{\Gamma_0}^* w})ds \\
&\quad\quad +\int_{\Gamma_0} (u-u_h)\cdot \ov{i\alpha w}ds\\
&= a(u-u_h,w) -i\alpha \int_{\Gamma_0} (u-u_h)\cdot \ov{w}ds \\
&\quad\quad +i\alpha \int_{\Gamma}\int_{\Gamma_0} (u(x)-u_h(x))^\top T_y u^*(x,y) \ov{w}(y)ds_yds_x\\
&\quad\quad -\int_{\Gamma}\int_{\Gamma_0} (u(x)-u_h(x))^\top T_x(T_yu^*(x,y))^\top \ov{w}(y)ds_yds_x\\
&= a(u-u_h,w) +b(u-u_h,w).
  \end{aligned}
 \end{equation*}
Noting the Galerkin orthogonality
\ben
a(u-u_{h},\xi)+b(u-u_{h},\xi)=0,\quad \forall \xi \in S_{h},
\enn
it follows that, for $\forall \xi \in S_{h}$,
\begin{equation}\label{Fin-equ15}
\begin{aligned}
&\quad a(u-u_{h},w)+b(u-u_{h},w)\\
&\lesssim \Vert u-u_{h}\Vert_{H^{1}(\Omega_{F})^\sigma} \inf_{\xi\in S_h}\Vert w-\xi\Vert_{H^{1}(\Omega_{F})^\sigma}  \\
&\lesssim h^{\alpha_{1}+\alpha_{2}}\Vert u\Vert_{H^{1+\alpha_1}(\Omega_{F})^\sigma}\Vert w\Vert_{H^{1+\alpha_{2}}(\Omega_{F})^\sigma}.
\end{aligned}
\end{equation}
Taking $e=u-u_{h}$ gives the final result and this completes the proof.
\end{proof}

\section{Numerical experiments}
\label{sec:5}
\noindent
This section presents a variety of numerical tests that demonstrate the accuracy and efficiency of the proposed scheme and we always set the impedance coefficient $\alpha=2$. Five different types of the boundary will be considered as follows:
\begin{align*}
\mathrm{circle (C)}:\;\; &x(\theta)=(\cos\theta,\sin\theta),\theta\in[0,2\pi),\\
\mathrm{ellipse (E)}:\;\; &x(\theta)=(3\cos\theta,2\sin\theta),\theta\in[0,2\pi),\\
\mathrm{kite (K)}:\;\;&x(\theta)=(\cos\theta+0.65\cos2\theta-0.65,1.5\sin\theta),\theta\in[0,2\pi),\\
\mathrm{peanut (P)}:\;\;&x(\theta)=(\sqrt{\cos^{2}\theta+0.25\sin^{2}\theta})(\cos\theta,\sin\theta),\theta\in[0,2\pi),\\
\mathrm{star (S)}:\;\; &x(\theta)=(1+0.3\cos5\theta)(\cos\theta,\sin\theta),\theta\in[0,2\pi),
\end{align*}
and we use $Y\star X$ to present that the boundary $\Gamma$ is of type $X$ and the artificial boundary $\Gamma_0$ is of type $Y$ with another appropriate size to enclose $\ov{\Omega}$ where $X,Y\in\{C, E, K, S, RT\}$.

{\bf Example 1.}(Laplace case) In the first example, we consider the problem of Laplace equation and the exact solution is given by
\begin{equation*}
u(x)=\frac{-x_{1}}{x_{1}^2+x_{2}^2}.
\end{equation*}
Fig.~\ref{NumExample1.1} shows the optimal convergence orders $\Vert u-u_{h}\Vert_{{L}^{2}(\Omega_F)}=\mathcal{O}(h^{2})$ and $\Vert u-u_{h}\Vert_{{H}^{1}(\Omega_F)}=\mathcal{O}(h)$ for the $C\star K$ and $C\star S$ boundaries.

\begin{figure}[htbp]
\center
\includegraphics[scale=0.6]{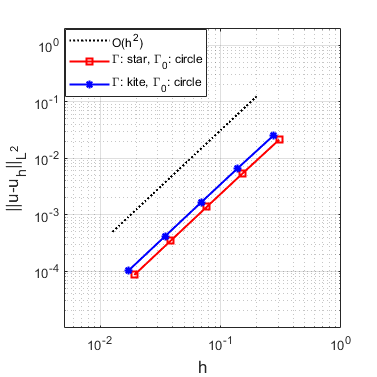}
\includegraphics[scale=0.6]{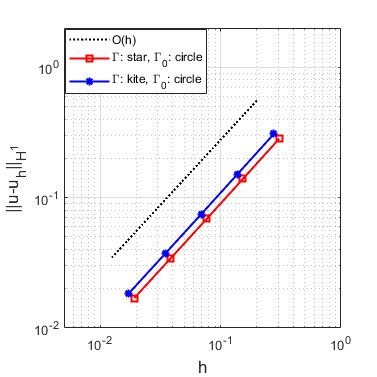}
\caption{Example 1. Log-log plot of the numerical errors in $L^{2}$-norm (left) and $H^{1}$-norm (right) for the $C\star K$ and $C\star S$ boundaries}
\label{NumExample1.1}
\end{figure}

{\bf Example 2.}(Acoustic case) In the second example, we consider the acoustic scattering problem and the exact solution is given by
\begin{equation*}
u(x)=H_{0}^{(1)}(k|x|).
\end{equation*}
Fig.~\ref{NumExample2.1} shows the optimal convergence orders $\Vert u-u_{h}\Vert_{{L}^{2}(\Omega_F)}=\mathcal{O}(h^{2})$ and $\Vert u-u_{h}\Vert_{{H}^{1}(\Omega_F)}=\mathcal{O}(h)$ for the case of $P\star S$ boundary and different wave numbers $k$. Then we consider the scattering of a plane incident wave $u^i=e^{ikx\cdot \bm{d}}$ with the propagation direction $\bm{d}=(1,0)^\top$. The numerical total field for the $K\star K$ boundary and $k=3$ are presented in Fig.~\ref{NumExample2.2}.

\begin{figure}[htbp]
\center
\includegraphics[scale=0.6]{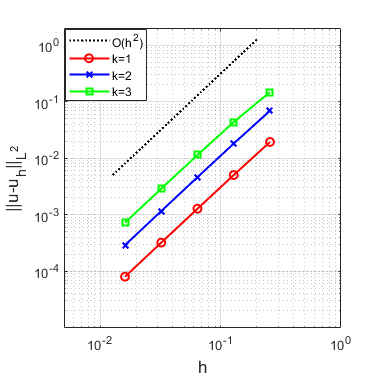}
\includegraphics[scale=0.6]{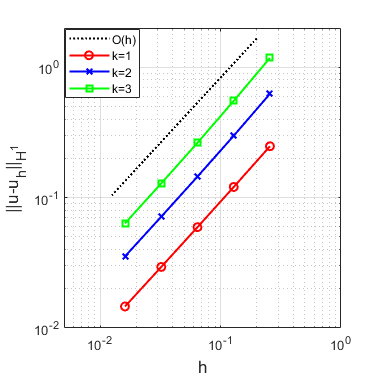}
\caption{Example 2. Log-log plot of the numerical errors in $L^{2}$-norm (left) and $H^{1}$-norm (right) for $P\star S$ boundary}
\label{NumExample2.1}
\end{figure}%

\begin{figure}[htbp]
\center
\includegraphics[scale=1.0]{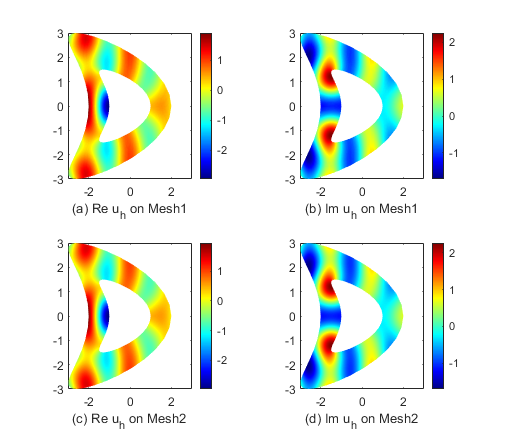}
\caption{Example 2. Numerical total field for the problem of acoustic scattering of a plane incident wave with $K\star K$ boundary. Mesh 1: 2288 triangles; Mesh 2: 146432 triangles}
\label{NumExample2.2}
\end{figure}

{\bf Example 3.}(Lam\'e case) Next, we consider the problem of Lam\'e equation and the exact solution is given by
\begin{equation*}
 \left\{
  \begin{aligned}
  &u_{1}(x, y)=-\left(\log r-\frac{(\lambda+\mu)(x_{1}-y_{1})^{2}}{r^{2}(\lambda+3\mu)}\right)\frac{\lambda+3\mu}{4\pi\mu(\lambda+2\mu)},\\
&u_{2}(x, y)=\frac{(\lambda+\mu)(x_{1}-y_{1})(x_{2}-y_{2})}{4\pi\mu r^{2}(\lambda+2\mu)},
  \end{aligned}
  \right.
\end{equation*}
where $r=\sqrt{(x_{1}-y_{1})^{2}+(x_{2}-y_{2})^{2}}$ and we choose $\mu=2$, $\lambda=3$, $(y_{1}, y_{2})=(0.5,0)^\top$. The numerical errors presented in Fig.~\ref{NumExample3.1} demonstrate the optimal convergence orders $\Vert u-u_{h}\Vert_{{L}^{2}(\Omega_F)^2}=\mathcal{O}(h^{2})$ and $\Vert u-u_{h}\Vert_{{H}^{1}(\Omega_F)^2}=\mathcal{O}(h)$ for the $C\star K$ and $E\star S$ boundaries.

\begin{figure}[htbp]
\begin{center}
     \includegraphics[width=6cm]{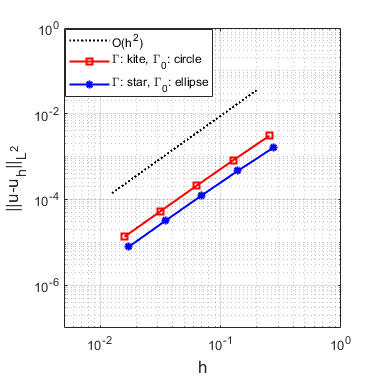}
     \includegraphics[width=6cm]{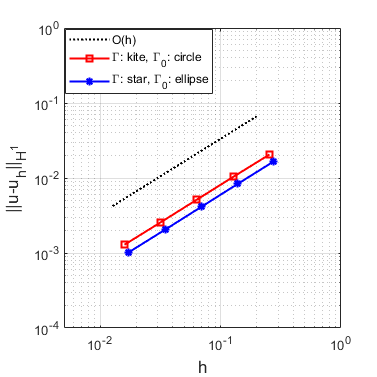}
\caption{Example 3. Log-log plot of the numerical errors in $L^{2}$-norm (left) and $H^{1}$-norm (right) for $C\star K$ and $E\star S$ boundaries}
\label{NumExample3.1}
\end{center}
\end{figure}

{\bf Example 4.}(Elastic case) Finally, we test the accuracy of the proposed scheme for solving the elastic scattering problem. Let the exact solution be given by
\begin{equation*}
u(x)=-\nabla H_{0}^{(1)}(k_{p}|x|),
\end{equation*}
and we set $\rho=0.5, \mu=2, \lambda=0.5$. Fig.~\ref{NumExample4.1} verifies the optimal convergence orders $\Vert u-u_{h}\Vert_{{L}^{2}(\Omega_F)^2}=\mathcal{O}(h^{2})$ and $\Vert u-u_{h}\Vert_{{H}^{1}(\Omega_F)^2}=\mathcal{O}(h)$ for the case of $C\star K$ boundary.

\begin{figure}[htbp]
\center
     \includegraphics[width=6cm]{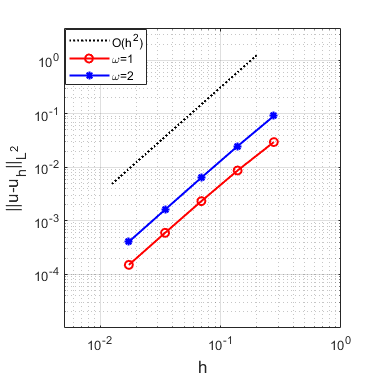}
     \includegraphics[width=6cm]{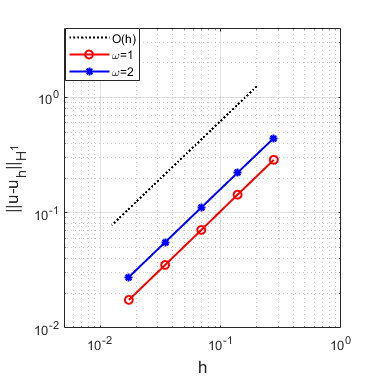}
\caption{Example 4. Log-log plot of the numerical errors in $L^{2}$-norm (left) and $H^{1}$-norm (right) for $C\star K$ boundary}
\label{NumExample4.1}
\end{figure}

\section*{Acknowledgments}
The work of M. Li is partially supported by the NSFC Grants 12271082 and 62231016.  T. Yin gratefully acknowledges support from NSFC through Grants 12171465 and 12288201.

\end{document}